%% file: p1m2_cr.tex
\documentclass[fleqn]{elsart}

\usepackage{amsfonts}
\usepackage{amssymb}
\usepackage{amsmath}
\usepackage{graphicx}

\def\la#1{\label{eq:#1}}
\def\re#1{(\ref{eq:#1})}

\newcommand{\reals}{\mathbb{R}}
\newcommand{\natnums}{\mathbb{N}}

\newcommand{\eqbydef}{\stackrel{\rm def}{=}}

\newcommand{\uo}[2]{\ensuremath{{u_{#1}}_{(#2)}}}

\newcounter{hypcglobal}
\newcounter{hypclocal} 
\newenvironment{hypolist}[1]%
{%
\renewcommand{\thehypclocal}{H\arabic{hypclocal}}
\begin{list}{\thehypclocal.}{\usecounter{hypclocal}}%
\setcounter{hypclocal}{#1}
}%
{%
\renewcommand{\thehypclocal}{\arabic{hypclocal}}
\setcounter{hypcglobal}{\thehypclocal}
\end{list}
}%

\begin{document}
\begin{frontmatter}

\title{A Geometric Approach to Feedback Stabilization of Nonlinear Systems with Drift\thanksref{project}}

\thanks[project]{This work was supported by NSERC of Canada under grant OGP-0138352.}

\author{H. Michalska}
\ead{michalsk@cim.mcgill.ca}
\author{\hspace{-.7ex}, }
\author{M. Torres-Torriti}
\ead{migueltt@cim.mcgill.ca}

\address{Department of Electrical \& Computer Engineering, McGill University,\newline    3480 University Street, Montr\'{e}al, QC, Canada H3A 2A7}

\begin{abstract}
The paper presents an approach to the construction of stabilizing feedback for strongly nonlinear systems.  The class of systems of interest includes systems with drift which are affine in control and which cannot be stabilized by continuous state feedback.  The approach is independent of the selection of a Lyapunov type function, but requires the solution of a nonlinear programming {\em satisficing problem} stated in terms of the logarithmic coordinates of flows.   As opposed to other approaches, point-to-point steering is not required to achieve asymptotic stability.  Instead, the flow of the controlled system is required to intersect periodically a certain reachable set in the space of the logarithmic coordinates.
\end{abstract}

\begin{keyword}
Nonlinear systems with drift; Control synthesis; Stabilization; Time-varying state feedback; Systems on Lie groups.
\end{keyword}

\end{frontmatter}


\section{Introduction}
The paper presents an approach to the design of feedback stabilizing controls for systems with drift which take the general form:
\begin{eqnarray}
\Sigma:\hspace{1cm}\dot{x}=f_0(x)+\sum_{i=1}^{m}f_i(x)u_i \eqbydef f^u(x)\la{S}
\end{eqnarray}
where, the state $x(t)$ evolves on $\reals^n$, $u\eqbydef[u_1,\ldots,u_m]$, $u_i\in\reals$ are the control inputs, $m<n$, and $f_i,\ i=0,1,\ldots,m$, are real analytic vector fields on $\reals^n$.  Feedback stabilization of systems of this type can be very challenging in the case when~\re{S} fails to satisfy Brockett's necessary condition for the existence of continuous stabilizing state feedback laws, see~\cite{BRO83}.  Systems of this type are encountered in a number of applications, for example, the control of rigid bodies in space, control of systems with acceleration constraints, and control of a variety of underactuated dynamical systems.  Hence, the development of {\em general} stabilization approaches for such systems deserves further attention.

As compared with driftless systems, relatively few approaches have so far been proposed for the stabilization of systems with drift.  The difficulty of steering systems with drift arises from the fact that, in the most general case of non-recurrent or unstable drift, the system motion along the drift vector field needs to be counteracted by enforcing system motions along adequately chosen Lie bracket vector fields in the system's underlying controllability Lie algebra.  Such indirect system motions are complex to design for and can be achieved only through either time-varying open-loop controls or discontinuous state feedback.  

The majority of relevant feedback stabilization methods for systems with drift  found in the literature, contemplate systems of specific structure and deliver feedback laws which apply exclusively to the particular models considered, see for example~\cite{BLO92,BUL00,GOD99,KOL94,MCL98,MOR97,PET99,REY99}.
 
More\hfill general\hfill feedback\hfill design\hfill methodologies\hfill have\hfill been\hfill proposed\hfill in\newline
 \cite{HER80,CRO84,LEO93,MIC97,MIC01} and apply to systems in general form~\re{S}.  The first systematic procedure for the construction of stabilizing piece-wise constant controls was presented in~\cite{HER80} and clearly demonstrates the difficulty of the problem.  The method of~\cite{HER80} has been applied in~\cite{CRO84} to asymptotically steer the the attitude and angular velocity of an underactuated spacecraft to the origin.  General averaging techniques on Lie groups have been successfully applied to attitude control in~\cite{LEO93}.  The Lie algebraic approach outlined in~\cite{MIC97} requires an analytic solution of a trajectory interception problem for the flows of the original system and its Lie algebraic extension in terms of the logarithmic coordinates on the associated Lie group.  Such analytic solutions are generally difficult to derive.  The approach in~\cite{MIC01} draws on the ideas of~\cite{COR91,COR92}, which apply to systems without drift.  The method of~\cite{MIC01} combines the construction of a periodic time-varying critically stabilizing control with the on-line calculation of an additional corrective term to provide for asymptotic convergence to the origin.  

In the above context, the contributions of this paper can be described as follows. 
\begin{itemize}
\item
An approach to stabilization of general systems of the form~\re{S} is presented which is based on the explicit calculation of a reachable set of desirable states for the controlled system.  Such reachable set is determined when the system $\Sigma$ is reformulated as a right-invariant system on an analytic, simply connected, nilpotent Lie group and once the stabilization problem is re-stated accordingly.  The reformulation allows for the time-varying part of the stabilizing feedback control to be derived as the solution of a nonlinear programming problem for steering the open-loop system $\Sigma$ to the given reachable set of states.  The construction of the feedback law does not require numerical integration of the model differential equation and is independent of the choice of a Lyapunov type function.

\item
Unlike other approaches,~\cite{HER80,CRO84,LAF93,MIC97,MIC01}, point-to-point steering is not required to achieve asymptotic stability.  Instead, the flow of the controlled system is required to intersect periodically a certain reachable set in the space of the logarithmic coordinates.  In contrast to~\cite{HER80,CRO84,LAF93}, the approach presented here offers a proof for Lyapunov asymptotic stability of the controlled system to the equilibrium and applies to systems whose linearization may be uncontrollable.

\item 
The approach presented might prove useful for the construction of feedback laws with a reduced number of control discontinuities and for the development of computationally feasible methods toward the design of smooth time-varying stabilizing feedback.  This is because the system is shown to remain asymptotically stable provided that the state of the system traverses the sets of desirable states periodically in time.  Since the sets of desirable states are typically large, the latter condition leaves much freedom for improved design, which is not the case for previous methods based on point-to-point steering.

\item
A fairly complex example is presented which clearly demonstrates the computational feasibility of the approach and its potential for extension to non-nilpotent systems.  The example employs a specialized Maple software package recently developed by the authors to facilitate symbolic Lie algebraic calculations\footnote{The {\em Lie Tools Package} (LTP) is freely available at: \newline
\texttt{http://www.cim.mcgill.ca/$\sim$migueltt/ltp/ltp.html}}.  
\end{itemize}

The\hfill proposed\hfill approach\hfill is\hfill based\hfill on\hfill a\hfill different\hfill concept\hfill than\hfill those\hfill used\hfill in\newline \cite{HER80,CRO84,LEO93}\hfill and\hfill compares\hfill favourably\hfill with\hfill the\hfill previous\hfill methods\hfill proposed\hfill in\newline \cite{MIC97,MIC01}.  One advantage of the present approach is that it does not require the exact analytic solution of the trajectory interception problem for the flows for the original and extended system, as in~\cite{MIC97}.  It also avoids the expensive on-line computation of the value and the gradient of the Lyapunov type function whose level sets are the trajectories of the critically stabilized system, as required in~\cite{MIC01}.  Furthermore, it should be easier to apply to systems with larger dimension of the state or controllability Lie algebra than the methods in~\cite{HER80,CRO84}.

\section{Problem Definition and Assumptions}\label{sec:def}
{\em {\bf Problem Definition.}  Construct a time-varying control law which globally stabilizes system $\Sigma$ to the origin.}

The sets of  positive integers and non-negative reals are denoted by $\natnums$ and $\reals_+$, respectively.  If $\mathcal{F}\eqbydef \{f_0,\ldots,f_m\}$ is a family of vector fields defined on $\reals^n$ then $L(\mathcal{F})$ is used to denote the Lie algebra of vector fields generated by $\mathcal{F}$, and $L_x(\mathcal{F})\eqbydef \{f(x)~|~f \in L(\mathcal{F}) \}\subset \reals^n$.  All vector fields considered here are assumed to be real, analytic, and complete, i.e. any vector field $f$ is assumed to generate a globally defined one-parameter group of transformations acting on $\reals^n$ and denoted by $\exp(t f)$.  This implies that for all $t_1, t_2\in\reals$, $\exp(t_1 f)\circ \exp(t_2 f)=\exp((t_1+t_2)f)$, and for all $x\in\reals^n$, $x(t)=\exp(t f)x$ satisfies the differential equation $\dot{x}=f$ with initial condition $x(0)=x$.  It is a well known fact, see~\cite[p. 95]{PAL57}, that if all generators in $\mathcal{F}$ are analytic and complete then all vector fields in $L(\mathcal{F})$ are analytic and complete. 

Solutions to system $\Sigma$, starting from $x(0)=x$ and resulting from the application of a control $u$, are denoted by $x^o(t,x,u)$, $t\geq 0$. 

The following hypotheses are assumed to hold with respect to system $\Sigma$ where $\mathcal{P}^m$ denotes the family of piece-wise constant functions, continuous from the right, and defined on $\reals^m$.

\begin{hypolist}{\thehypcglobal}
\item\label{h:complete-nilpotent}
The vector fields $f_0,\ldots,f_m : \reals^n\rightarrow\reals^n$ are real, analytic, complete, and linearly independent, with $f_0(0)=0$, and generate a nilpotent Lie algebra of vector fields $L(\mathcal{F})$, such that its dimension is $\mbox{dim} L(\mathcal{F}) = r \geq n+1$.
\item\label{h:controllable}
The system $\Sigma$ is strongly controllable, i.e. for any $T>0$ and any two points $x_0, x_f \in \reals^n$, $x_f$ is reachable from $x_0$ by some control $u \in \mathcal{P}^m$ of $\Sigma$ in time not exceeding $T$; i.e. there exists a control $u \in \mathcal{P}^m$ and a time $t \leq T$ such that $x^o(t,x_0,u)=x_f$.
\end{hypolist}

The method presented here employs an arbitrary Lyapunov type function $V:\reals^n\rightarrow \reals_+$ which is only required to satisfy the following conditions:
\stepcounter{hypcglobal}
\begin{hypolist}{\thehypcglobal}
\item[H{\thehypcglobal}.a.]~\label{h:vdiff-vpos} $V$ is twice continuously differentiable with $V(0)=0$. Additionally, there exists a constant  $\zeta>0$, such that for all $x\in\reals^n$, $\|\nabla{V}(x)\|\geq \zeta\|x\|$.
\item[H{\thehypcglobal}.b.]~\label{h:vdec}  $V$ is positive definite and decrescent, i.e. there exist continuous, strictly increasing functions $\alpha(\cdot):\reals_+\rightarrow\reals_+$ and $\beta (\cdot):\reals_+\rightarrow\reals_+$, 
with  $\alpha(0)=\beta(0)=0$, such that for all $x\in\reals^n$, $\alpha(\|x\|)\leq V(x) \leq \beta(\|x\|)$.  
\end{hypolist}
The function $V$ is not a Lyapunov function in the usual sense in that it will be allowed to increase instantaneously along the trajectories of the stabilized system.  However, in the sequel, the function $V$ will still be referred to as a Lyapunov function.

\section{Basic Facts and Preliminary Results}
Let $\mathcal{R}_\mathcal{F}(T,x)$ denote the reachable set of $\Sigma$ at time $T$ from $x$ by piece-wise constant controls.  A well known consequence of strong controllability is that the condition for accessibility is satisfied, namely: $\mbox{span}\,L_x(\mathcal{F})=\reals^n$ for all $x\in\reals^n$.

It is helpful to introduce the following subsets of $\textit{diff}(\reals^n)$, the group (under composition) of diffeomorphisms on $\reals^n$:
\newline
{\small
\hspace*{-0mm}
\parbox[t]{\textwidth}{
\begin{eqnarray} 
&G   \eqbydef \{ \exp( t_1 f^{\uo{}{1}}) \circ \cdots \circ \exp( t_k f^{\uo{}{k}})~|~\uo{}{i} \in \reals^m;\ t_i \in \reals;\ 
k\in\natnums \}&\mbox{\hspace{5mm}}\nonumber \\
&G_T \eqbydef \{ \exp( t_1 f^{\uo{}{1}}) \circ \cdots \circ \exp( t_k f^{\uo{}{k}})~|~\uo{}{i} \in \reals^m;\ t_i \geq 0;\  
\sum_{i=1}^k t_i = T;\ k\in\natnums \}&\mbox{\hspace{5mm}}\nonumber
\end{eqnarray} 
}
}
\newline
where $f^{\uo{}{i}}\eqbydef f_0+\sum_{j=1}^{m}f_j\,\uo{j}{i}$, and $\uo{j}{i}$ are components of $\uo{}{i}$.  By the strong controllability hypothesis, if $G\,x$ and $G_T\,x$ denote the orbits through $x$ of $G$ and $G_T$, respectively, i.e. $G\,x = \{ g\,x~|~g \in G \}$ and $G_T\,x = \{g\,x~|~g \in G_T \}$, then $\mathcal{R}_\mathcal{F}(T,x) = G_T\,x=G\,x =\reals^n$, for any $T>0$.  The group $G \subset \textit{diff}(\reals^n)$ is a subgroup of $\textit{diff}(\reals^n)$,~\cite{HIR75}.  Moreover, by virtue of the results by R. Palais,~\cite{PAL57}, $G$ can be given a structure of a Lie group with Lie algebra isomorphic to $L(\mathcal{F})$.  The result of Palais~\cite{PAL57}, interpreted so as to apply to system $\Sigma$, is worth citing:
%
\begin{thm}[Palais,~\cite{PAL57}, p. 95]
Let $L(\mathcal{F})$ be a finite dimensional Lie algebra of vector fields defined on $\reals^n$.  Assume that all the generators in $\mathcal{F}$ are analytic and complete.  Then there exists a unique analytic, simply connected Lie group $H$, whose underlying group is a subgroup of $\textit{diff}(\reals^n)$, and a unique global action of this group on $\reals^n$ defined as an analytic mapping $\phi : H  \times \reals^n \ni (h,p) \rightarrow h(p) \in \reals^n$ which induces an isomorphism between the Lie algebra, $L(H)$, of right invariant vector fields on $H$, and $L(\mathcal{F})$.  Precisely, the isomorphism $\phi^+_L : L(H) \rightarrow L(\mathcal{F})$ is constructed by setting:
\[
\phi^+_L(\lambda)(p) =(d\phi_p)_e(\lambda(e))\ \ \mbox{for all } 
\lambda \in L(H),\ p \in \reals^n 
\]
where $e \in H$ is the identity element, and $(d\phi_p)_e$ is the differential of a mapping $\phi_p: H \rightarrow \reals^n$ at identity.  For any $p \in \reals^n$, the mapping $\phi_p$ is defined by: $\phi_p(h) = h(p)$.
\end{thm}
This result has strong implications.  It can be shown, see~\cite[Thm. 3.2.]{HIR75}, that the isomorphism $\phi^+_L$ induces an isomorphism, denoted by $\phi^+_G$, between the groups $H$ and $G$.  If for a constant control $u \in \reals^m$, $\lambda\in L(H)$ and $f^u$ are related by $\phi^+_L(\lambda)=f^u$, where $f^u$ is the right-hand side of~\re{S}, then the one-parameter group $\exp(t \lambda)$ maps into the one-parameter group $\exp(tf^u)$ as follows, see~\cite[Thm. 2.10.3, p. 91]{VAR84}:
\[
\phi^+_G \left (\exp(t \lambda)\right )=\exp(t \phi^+_L(\lambda))=\exp (tf^u)\ \ \mbox{for all } t \geq 0 
\]
Each element $\exp(t_1 f^{\uo{}{1}}) \circ \cdots \circ \exp(t_k f^{\uo{}{k}})\in G$ can thus be expressed as $\phi^+_G \{\exp(t_1 \lambda_1) \circ \cdots \circ \exp(t_k\lambda_k)\}$ where: $\phi^+_L (\lambda_i)=f^{\uo{}{i}}$ for all $i=1,\ldots,k$.  It follows that $H$ is given by:
\begin{eqnarray}
\lefteqn{\hspace{-6mm}H=(\phi^+_G)^{-1}(G)=}\nonumber\\
  &&\hspace{-3mm}\{\exp(t_1 \lambda_1) \circ \cdots \circ \exp(t_k \lambda_k)~|~\lambda_i \in L(H),\ \phi^+_L(\lambda_i) \in L(\mathcal{F});\ t_i \in \reals;\ 
k\in\natnums \}\nonumber
\end{eqnarray} 
The above facts allow to reformulate the system $\Sigma$ as a right invariant system $\Sigma_H$ evolving on the Lie group $H$ as follows.  If the right invariant vector fields $\eta_i\in L(H)$ are such that $\phi^+_L (\eta_i)=f_i$ for all $i=0,\ldots,m$, then
\begin{eqnarray}
\Sigma_H:\hspace{1cm}\dot{S}(t)= \{ \eta_0+\sum_{i=1}^{m} u_i \eta_i \}S(t)\hspace{5mm} \mbox{ with } S(0) =e;\ \  t \geq 0 \la{SonH}
\end{eqnarray}
The simplified notation used in the above expression for $\Sigma_H$ deserves explanation.  If $TH_h$ denotes the tangent space to $H$ at $h \in H$, then for any $\eta \in TH_e$,  the expression $\eta S(t)$ denotes the image of $\eta$ under the map $dR_{S(t)} : TH_e \rightarrow TH_{S(t)}$ induced by the map of right translation by $S(t) \in H$: $h \rightarrow h S(t)$, for all $h \in H$.  In this notation, if $\eta$ is represented by the curve $h(\tau) \in H$, $\tau\geq 0$, then $\eta S(t)$ is represented by the curve $h(\tau)S(t) \in H$, $\tau\geq 0$.  Under the assumptions made:
\begin{prop}
~\label{prop:1}
The system $\Sigma_H$ is strongly controllable on $H$ from the identity $e$, i.e. given any $T > 0$, any $h\in H$ is reachable from $e\in H$ by a trajectory of $\Sigma_H$ using a control $u\in \mathcal{P}^m$, in time not exceeding $T$.  There exists a diffeomorphism between trajectories of systems $\Sigma$ and $\Sigma_H$ in the sense that: if $S(t)$, $t\in [0,T]$, is a trajectory of $\Sigma_H$ through $e$, corresponding to a concatenation $u \in \mathcal{P}^m$ of constant controls $\uo{}{i} \in \reals^m$ defined on intervals of lengths $t_i$, $i=1,\ldots,j$, respectively, then $x^o(t,x,u)= \phi^+_G(S(t))x$, $t \in [0,T]$, is a trajectory of $\Sigma$ through the point $x$ corresponding to the same piece-wise constant control $u$.
\end{prop}
\begin{pf}
The\hfill image\hfill of\hfill $H_T\hfill \eqbydef\hfill \{\exp(t_1 \lambda_1) \circ \cdots \circ \exp(t_k \lambda_k)\hfill~|~\hfill\lambda_i\in L(H),\newline \phi^+_L(\lambda_i)\in L(\mathcal{F});\ t_i \in \reals;\ \sum_{i=1}^k t_i = T,\ k\in\natnums \}$ under the group isomorphism $\phi^+_G$ is $G_T$.  Hence $H = (\phi^+_G)^{-1} (G)=(\phi^+_G)^{-1} (G_T)= H_T$, by strong controllability of $\Sigma$.  It follows that system $\Sigma_H$ is strongly controllable from the identity because, for any $T>0$, $\mathcal{R}_\mathcal{H}(T,e)=H_T\,e=H\,e=H$ where $\mathcal{R}_\mathcal{H}(T,e)$ is the reachable set from identity for system $\Sigma_H$ at time $T$, and $\mathcal{H}\eqbydef \{\eta_0,\ldots,\eta_m\}$.  To show that $x^o(t,x,u)$, $t\in [0,T]$, is a trajectory of $\Sigma$ it suffices to notice that, since $\phi^+_G$ is a group isomorphism corresponding to a Lie algebra isomorphism $\phi^+_L$, 
\begin{eqnarray}
\phi^+_G \left (\exp(t_1 \lambda_1) \circ \cdots \circ \exp(t_j \lambda_j)\right )\,x
 &=& \exp(t_1 \phi^+_L(\lambda_1)) \circ \cdots \circ \exp(t_j \phi^+_L(\lambda_j))\,x\hspace{7mm}
 \nonumber\\
 &=& \exp(t_1 f^{\uo{}{1}}) \circ \cdots \circ \exp( t_j f^{\uo{}{j}})\,x
 \nonumber
\end{eqnarray}
~\vspace*{-5ex}\\
\hspace*{\fill}\qed\end{pf}
For steering purposes, it is useful to consider the Lie algebraic extension of $\Sigma$,  $\Sigma^e$, defined as
\begin{eqnarray}
\Sigma^e:\hspace{1cm}\dot{x}=g_0(x)+\sum_{i=1}^{r-1}g_i(x)v_i \eqbydef g^{v}(x) \la{Se}
\end{eqnarray}
where, for simplicity of exposition, the set $g_i$, $i=1,\ldots,r-1$, is assumed to contain a basis for $L_x(\mathcal{F})$ for any $x$ in a sufficiently large neighbourhood of the origin $B(0,R)$, and where $\mathcal{G}\eqbydef \{g_0,\ldots,g_{r-1}\}$ is a basis for $L(\mathcal{F})$, and ${v}\eqbydef[v_1,\ldots,v_{r-1}]$.

Additionally, it is assumed that $g_i = f_i$ for $0\leq i\leq m$, and the remaining vector fields $g_i$, $m+1\leq i\leq r-1$, are some Lie brackets of the system vector fields $f_i$.  Let the solution of $\Sigma^e$ through a point $x$, and under the action of a control ${v}$, be denoted by $x^e(t,x,{v})$, $t \geq 0$.

By construction, both systems $\Sigma$ and $\Sigma^e$ have the same underlying algebra of vector fields $L(\mathcal{F})$.  Additionally, system $\Sigma^e$ not only inherits the strong controllability property of $\Sigma$, but is in fact instantaneously controllable in any direction of the state space.  It is hence much easier to steer than $\Sigma$ whose motion along some directions has to be generated by time-varying controls, or else piece-wise constant switching controls.  Similarly to $\Sigma$, the system $\Sigma^e$ can be reformulated on the Lie group $H$ as
\begin{eqnarray}
\Sigma^e_H:\hspace{1cm}\dot{S^e}(t)= \{ \mu_0+\sum_{i=1}^{r-1} {v}_i \mu_i \}S^e(t) \hspace{5mm} \mbox{ with } S^e(0) =e;\ \  t \geq 0 \la{SeonH}
\end{eqnarray}
with $\phi^+_L (\mu_i)=g_i$ for all $i=0,\ldots,r-1$.  Clearly, Proposition~\ref{prop:1} also holds for $\Sigma^e_H$; i.e. system $\Sigma^e_H$ is strongly controllable from identity, and trajectories of $\Sigma^e_H$ map into trajectories of $\Sigma^e$ according to: $x^e(t,x,{v})= \phi^+_G(S^e(t))x$, $t \in [0,T]$.

Since $H$ is analytic, simply connected, and nilpotent ($L(H)$ is nilpotent), it follows that the exponential map on $L(H)$ is a global diffeomorphism onto $H$, see~\cite[Thm. 3.6.2., p. 196]{VAR84}.  Let $\{\psi_0,\ldots,\psi_{r-1}\}$ be a basis of $L(H)$ which, without the loss of generality, is ordered in such a way that $\psi_i=\eta_i=\mu_i$, $i=0,\ldots,m$, and is such that the mapping
\begin{eqnarray}
\reals^r\ni (t_0,\ldots,t_{r-1})\rightarrow \exp(t_0\,\psi_0)\circ\cdots\circ \exp(t_{r-1}\,\psi_{r-1})\in H \la{globcoordonH}
\end{eqnarray}
is a global coordinate chart on $H$; see Remark~\ref{bfpr:rem1} on how to construct one such basis.  This implies that the solutions to $\Sigma_H$ and $\Sigma_H^e$, whose common underlying Lie algebra is $L(H)$, can be expressed as products of exponentials:
\begin{eqnarray}
S(t)  =\prod_{i=0}^{r-1}\exp(\gamma_i(t)  \psi_i)\mbox{ and }
S^e(t)=\prod_{i=0}^{r-1}\exp(\gamma_i^e(t)\psi_i)
 \la{prodexp}
\end{eqnarray}
where the functions $\gamma_i,\gamma_i^e:\reals\rightarrow\reals$, are referred to as the $\gamma$-coordinates of the flows $S(t)$, $S^e(t)$, respectively, and can be shown to satisfy a set of differential equations of the form, see~\cite{WEI64}:
\begin{eqnarray}
\Gamma(\gamma)  \dot{\gamma  }= u^d \hspace{10mm}
\Gamma(\gamma^e)\dot{\gamma^e}={v^d}\hspace{10mm}\gamma(0)=\gamma^e(0)=0
 \la{geq}
\end{eqnarray}
Here $\Gamma(\cdot):\reals^r\rightarrow\reals^{r\times r}$ is a real analytic, matrix valued function of $\gamma\eqbydef [\gamma_0,\ \ldots,\ \gamma_{r-1}]^T$, the zero initial conditions correspond to $S(0)=S^e(0)=I$, and $u^d\eqbydef [1\ u_1\ \ldots\ u_m\ 0 \ldots 0]\in\reals^{r}$, and $v^d\eqbydef [1\ v_1\ \ldots\ v_{r-1}]\in \reals^{r}$, (where the first component of $u^d$ and ${v}^d$ corresponds to the drift vector field).

The solution to equation~\re{geq} is generally only local unless $\Gamma(\gamma)$ is invertible for all $\gamma$.  The invertibility of $\Gamma(\gamma)$ is ensured if a basis for $L(H)$ is constructed as indicated in Remark~\ref{bfpr:rem1}; see also~\cite{WEI64}.    

\begin{rem}
\label{bfpr:rem1}
Since $L(H)$ is the Lie algebra of the analytic, simply connected, nilpotent Lie group $H$, then $L(H)$ is solvable and thus there exists a chain of ideals $0\subset \mathcal{I}_{r-1} \subset \mathcal{I}_{r-2} \subset \cdots \subset \mathcal{I}_0 = L(H)$ where each $\mathcal{I}_i$ is exactly of dimension $r-i$. Without the loss of generality a basis $\{\psi_0,\ldots,\psi_{r-1}\}$ for $L(H)$ can easily be constructed so that each $\mathcal{I}_i$ is generated by $\{\psi_{i},\ldots,\psi_{r-1}\}$.  As the exponential map $\exp:L(H)\rightarrow H$ is a global  analytic diffeomorphism onto analytic, simply connected, nilpotent Lie groups, then any element $h\in H$ has a unique representation as a single exponential:
\[
h=\exp\left (\sum_{i=0}^{r-1}\theta_i  \psi_i\right ) 
\]
where $\theta_i$, $i=0,1,\ldots,r-1$, are known as Lie-Cartan coordinates of the first kind.  This exponential can further be written as a product of exponentials by showing that the mapping~\re{globcoordonH} is onto and one-to-one, thus making it a global chart for $H$.  The fact that, in the basis chosen above,~\re{globcoordonH} is a bijection would require proof.  This proof is omitted here for reasons of brevity (see Lemma on page 326 of~\cite{SUS86}, for a similar argument).  An algorithmic way to obtain a basis for $L(H)$ which satisfies the above mentioned conditions is to employ the construction procedure given by P. Hall~\cite{SER65}.
\end{rem}
%
\begin{rem}
\label{bfpr:rem2}
The representation~\re{prodexp} of the solution to equation~\re{SonH} is not unique.  The representation in the form of the product of exponentials~\re{prodexp} results from the introduction of the Lie-Cartan coordinates of  the second kind~\re{globcoordonH} on the group $H$; see~\cite{WEI64}.  Alternatively, as pointed out in Remark~\ref{bfpr:rem1}, the solution to~\re{SonH} can be represented using the Lie-Cartan coordinates of the first kind, i.e. it is possible to write
\[
S(t)  =\exp\left (\sum_{i=0}^{r-1}\theta_i(t)  \psi_i\right )
\]
where $\theta_i:\reals\rightarrow\reals$, $i=0,1,\ldots,r-1$, are the ``coordinates'' of such a solution; see~\cite{MAG54}.

It is also known that the solution of~\re{SonH} can be written in terms of a formal series of $\mathcal{C}^\infty$ functions on $H$.  The last arises when~\re{SonH} is solved by Picard iteration giving rise to the Peano-Baker formula which exhibits the solution in terms of iterated integrals.  Specifically, $S(T)$ defines the Chen-Fliess series of the input $u$; see~\cite[Theorem III.2, p. 22]{FLI83} and~\cite[p. 695]{SUS83}.  The equivalence of the Chen-Fliess series representation and the representation through the product of exponentials of the solutions to~\re{SonH} has been shown in~\cite{SUS86} (see Theorem on p. 328).  The product expansion in~\cite{SUS86} uses P. Hall bases.
\end{rem}

\section{Stabilizing Feedback Control Design}
Generally, a control Lyapunov function which leads to a smooth feedback control law for system $\Sigma$ is not guaranteed to exist.  An alternative idea for the construction of a stabilizing feedback relies on achieving a periodic decrease in an arbitrarily imposed Lyapunov type function $V$, through the action of a time-varying control.  Periodic decrease is defined in terms of the condition $V(x(t_0+nT))-V(x(t_0+(n-1)T))<0$, which is required to hold for all $n \in \natnums$ and where $x(t)$, $t\geq t_0$, is the trajectory of the closed-loop system and $T>0$ is the ``period'' of decrease.  

To construct a feedback control which guarantees a periodic decrease in $V$, it is convenient to re-state the stabilization problem on the Lie group $H$.  Since $\Sigma^e$ is instantaneously controllable in any direction of the state space, it is helpful to employ the extended system as the first instrument to achieve such a decrease.  To this end, for any $x\in B(0,R)$ let a set $U^e(x)$ of admissible extended controls be introduced as follows:
\[
U^e(x)\eqbydef \left \{{v}\in\reals^{r-1}~|~\nabla{V}g^{v}(x)
                  <-\eta\|x\|^2,\ \|{v}\|\leq M\|x\|\right \} 
\]
where the constant $R>0$ is sufficiently large to accommodate for all initial conditions of interest, and $M>0$ is to be chosen later.  The set $U^e(x)$ translates into a reachable set of states of the extended system $\Sigma^e$ at time $T$, $\mathcal{R}_\mathcal{G}(T,x,U^e(x))$:
\[
\mathcal{R}_\mathcal{G}(T,x,U^e(x))\eqbydef\{z\in \reals^n~|~z=x^e(T,x,{v}),\ {v}\in U^e(x)\} 
\]
where $x^e(T,x,{v})$ denotes the trajectory of $\Sigma^e$ emanating from $x$ at time $t=0$ and resulting from the application of the control ${v}$ over the time interval $[0, T]$.  The reason for introducing $U^e(x)$ is explained in terms of the following result.  

\begin{prop}~\label{prop:2}
Under hypotheses~\ref{h:complete-nilpotent}--\mbox{H\thehypcglobal} there exists a time horizon $T_{max}>0$ such that for all $t>0$ and $T\in [0,T_{max}]${\em :}
\begin{eqnarray}
V(z) - V(x) \leq -\frac{\eta}{2}\|x\|^2 T \la{decVe}
\end{eqnarray}
for all $z\in\mathcal{R}_\mathcal{G}(T,x,U^e(x))$.
\end{prop}
\begin{pf}
See Appendix~\ref{appen1}.
\end{pf}

In this context it is desirable to construct an open loop control for system $\Sigma$ which guarantees an equivalent decrease in the Lyapunov function as stated by~\re{decVe}.  This can be achieved by the construction of any $\bar{u}\in\mathcal{P}^m$ which ensures:
\begin{eqnarray}
x^o(T,x,\bar{u})\in \mathcal{R}_\mathcal{G}(T,x,U^e(x)) \la{rtic}
\end{eqnarray} 
The above is a control problem for which the terminal constraint set has no direct characterization.  However, when~\re{rtic} is re-stated on the Lie group $H$ it translates into a computationally feasible nonlinear programming problem which can be formulated in terms of the $\gamma$-coordinates. This is done as follows.

By virtue of the definitions in Section~\ref{sec:def} the reachable set $\mathcal{R}_\mathcal{G}(T,x,U^e(x))$ is the orbit $G(T,U^e(x))\,x =\mathcal{R}_\mathcal{G}(T,x,U^e(x))$ where
\[
G(T,U^e(x)) \eqbydef \{\exp(Tg^{v})~|~{v}\in U^e(x)\}\subset G
\]
Also 
\[
H(T,U^e(x)) \eqbydef \{S^e(T,{v}) ~|~{v}\in U^e(x)\}=(\phi^+_G)^{-1}\left (\rule[-0.05em]{0em}{1.1em} G(T,U^e(x))\right )\subset H
\]
where $S^e(T,{v})$ denotes the value of the solution to equation~\re{SeonH} at time $T$ and due to extended control ${v}$.  When expressed in the global coordinate system,~\re{globcoordonH}, each element of $H(T,U^e(x))$ has the representation:
\begin{eqnarray}
S^e(T,{v}) = \prod_{i=0}^{r-1}\exp(\gamma_i^e(T,{v}^d) \psi_i) \la{flow:S^e}
\end{eqnarray} 
where $\gamma^e(T,{v}^d) \eqbydef [\gamma_0^e, \ldots, \gamma_{r-1}^e](T,{v}^d)$, is the value of the solution to equation~\re{geq} at time $T$ and due to control ${v}^d=[1\ {v}]^T$ with ${v}\in U^e(x)$.  Since 
\[
x^o(T,x,\bar{u})=\phi_G^+\left (\prod_{i=0}^{r-1} \exp(\gamma_i(T,\bar{u}^d) \psi_i) \right )x
\]
where $\gamma(T,\bar{u}^d) \eqbydef [\gamma_0, \ldots, \gamma_{r-1}](T,\bar{u}^d)$ is the value of the solution to equation~\re{geq} at time $T$ and due to control $\bar{u}^d=[1\ \bar{u}\ 0\ \ldots\ 0]^T\in\mathcal{P}^{r}$ with $\bar{u}\in\mathcal{P}^m$, then~\re{rtic} holds if
\[
\prod_{i=0}^{r-1} \exp(\gamma_i(T,\bar{u}^d) \psi_i) \in H(T,U^e(x))
\]
Due to the representation~\re{flow:S^e}, it hence follows that~\re{rtic} holds if 
\[
\gamma(T,\bar{u}^d)\in \mathcal{R}_{\gamma}(T,U^e(x))
\]
where
\[
\mathcal{R}_{\gamma}(T,U^e(x)) \eqbydef \{\gamma^e(T,{v}^d)~|~{v}\in U^e(x)\}
\]
For any {\em constant} control ${v}^d=[1,v]\in \reals^r$ equation~\re{geq} can be integrated symbolically to yield $\gamma^e(T,{v}^d)=\int_{0}^{T}\Gamma^{-1}(\gamma^e(\tau,{v}^d)){\d}\tau\,{v}^d\eqbydef M(T)\,{v}^d$.  Since $\Gamma$ is nonsingular and triangular for any of its arguments $\gamma^e$, then $M(T)$ is also nonsingular and triangular for any integrated trajectory $\gamma^e(\tau,{v}^d)$, $\tau\in[0,T]$, with a constant ${v}^d$.  Since one of the components of $v^d$ is equal to one (due to the presence of the drift vector field), an {\em analytic expression} can be derived for the inverse mapping $F:\reals^{r}\times\reals\rightarrow\reals^{r-1}$ such that ${v}=F(\gamma^e(T,{v}^d),T)$.

It follows that
\begin{eqnarray}
\mathcal{R}_{\gamma}(T,U^e(x))& = & \{\gamma\in\reals^r~|~F(\gamma,T)\in U^e(x)\} \la{reach_gamma}
\end{eqnarray}
which is given explicitly and permits to express~\re{rtic} as a nonlinear programming problem with respect to the variables parametrizing the piece-wise constant control $\bar{u}$.  Assuming that such a parametrization is given by  $\uo{}{k}$, $k=1,\ldots,s$, $s\in\natnums$, so that 
\[
  \bar{u}(\tau)\ \eqbydef\ 
\uo{}{k},\ \   
\tau\in[t_k,t_k+\varepsilon),\ \ \varepsilon s=T\mbox{ and }
t_1=0,\ t_k=t_{k-1}+\varepsilon,\ k=1,2,\ldots,s 
\]
the solution of $\re{geq}$, $\gamma(T,\bar{u}^d)$, is also parametrized by $\uo{}{k}$, $k=1,\ldots,s$.

The nonlinear programming problem equivalent to~\re{rtic} is hence stated as the following {\em satisficing problem} (SP):

\hspace{1cm}{\bf SP:} \parbox[t]{.8\textwidth}{
For given constants $\eta >0$, $T>0$ and $M>0$, and for $x \in B(0,R)$, find feasible parameter vectors $\uo{}{k}$, $k=1,\ldots,s$, such that:}
\begin{eqnarray}
\hspace*{2cm}\gamma(T,\bar{u}^d)\in\mathcal{R}_{\gamma}(T,U^e(x)) \la{sp}
\end{eqnarray}
Concerning the selection of the constant $M$ and the existence of solutions to SP, it is possible to show the following.

\begin{prop}~\label{prop:3} 
Under assumptions~\ref{h:complete-nilpotent}--\mbox{H\thehypcglobal}, for any neighbourhood of the origin $B(0,R)$ and any $\eta>0$, there exists a constant  $M(R,\eta)>0$, such that a solution to SP exists for any $x \in B(0,R)$, and any control horizon $T>0$, provided that $s\in\natnums$, the number of switches in the control sequence $\bar{u}$, is allowed to be large enough.
\end{prop}
\begin{pf}
See Appendix~\ref{appen2}.
\end{pf}

\begin{prop}~\label{prop:4} 
Let $\bar{u}(x,\tau)$, $\tau \in [0,T]$, be a control generated by the solution $\bar{u}$ to SP.  There exists a $T_{max}>0$ such that for all $T\in [0, T_{max}]${\em :}
\begin{eqnarray} 
V(x^o(T,x,\bar{u}))-V(x)< -\frac{\eta}{2}\|x\|^2 T \la{decVo2}
\end{eqnarray}
\end{prop}
\begin{pf}
If $\bar{u}$ solves SP then $\gamma(T,\bar{u}^d)\in \mathcal{R}_{\gamma}(T,U^e(x))$. Hence, there exists a control ${v}\in U^e(x)$ such that $x^o(T,x,\bar{u})=x^e(T,x,{v}) \in \mathcal{R}_\mathcal{G}(T,x,U^e(x))$ which proves~\re{decVo2} by virtue of Proposition~\ref{prop:2}.
\hspace*{\fill}\qed\end{pf}

The results presented above now serve for the construction of the stabilizing feedback.

\section{The Stabilizing Feedback and its Analysis}
The stabilizing feedback control, $u^c(x,\tau)$, $\tau\geq 0$, $x\in B(0,R)$ for system $\Sigma$ is defined as a concatenation of solutions to SP, $\bar{u}(x(nT),\tau)$, $\tau \in [nT, (n+1)T]$, computed at discrete instants of time $nT$, $n\in\natnums\cup \{0\}$:
\begin{eqnarray}
u^c(x, \tau)\eqbydef \bar{u}(x(n\,T), \tau) \mbox{  for all } \tau \in 
  [nT, (n+1)T], \ \ \ n\in\natnums \cup \{0\} \la{contr}
\end{eqnarray}
where $x(nT)$ is the state of the closed-loop system $\Sigma$ at time $nT$.

\begin{rem}~\\[-2em]
\begin{itemize}
\item 
The concatenated control $u^c(x,t)$ is a feedback control in the sense that a solution to SP is computed at each $t=nT$, $n\in\natnums\cup \{0\}$ and thus depends on $x(nT)$.
\item 
An off-line construction of the feedback law could possibly be envisaged in that the satisficing problem could be solved on a finite collection of compact non-overlapping subsets $C_s$ covering $B(0,R)$.  The objective function for the  SP should for this purpose be modified as follows:
\[
\gamma(T,\bar{u}^d)\in\bigcup_{x\in C_s} \mathcal{R}_{\gamma}(T,U^e(x)) 
\]
\item 
The computation of the analytic expression for the mapping $F$ defining the reachable set $\mathcal{R}_{\gamma}(T,U^e(x))$ in~\re{reach_gamma} can be facilitated by adequate supporting software for symbolic manipulation of Lie algebraic expressions.  Such software has been developed by the authors (see ~\cite{TOR02}) in the form of a software package for Lie algebraic computations in Maple which can be used here for the construction of a basis for the controllability Lie algebra, simplification of arbitrary Lie bracket expressions, and the derivation of the equation for the evolution of the $\gamma$-coordinates.  
\end{itemize}
\end{rem}

\begin{thm}
Let $T\in[0,T_{max}]$, where $T_{max}$ is specified in Proposition~\ref{prop:2} and let the constant $M$ be selected as in Proposition~\ref{prop:3}.  Suppose there exists a constant $C>0$ such that the solutions to SP are bounded as follows
\begin{eqnarray}
\|\bar{u}(x(nT),\tau)\| \leq C\|x(nT)\|\ \ \mbox{for all } \tau \in [0,T],\ n \in \natnums \cup \{0\} \la{ubarbnd} 
\end{eqnarray}
Under these conditions, the concatenated control $u^c(x,\tau)$ given by~\re{contr} renders the closed-loop system $\Sigma$ uniformly asymptotically stable.
\end{thm}
\begin{pf}
Let $t_k=t_0+kT$,  $k\in\natnums$, and let $x(t)$ denote the state of the closed-loop system at time $t$ due to control~\re{contr}.  By Proposition~\ref{prop:4} the state of system $\Sigma$ with control input~\re{contr} satisfies 
\begin{eqnarray}
V(x(t_{k+1}))-V(x(t_k))\leq-\varphi(\|x(t_k)\|) \hspace{3em} \forall\ k\in\natnums \cup \{0\} \la{vadec}
\end{eqnarray}
where $\varphi(\|x(t_k)\|)=\frac{\eta}{2}\|x(t_k)\|^2 T$.

By invoking the Gronwall-Bellman lemma with $\bar{u}$ satisfying~\re{ubarbnd} it is easy to show that (see the proof of Proposition~\ref{prop:2} in which the state of $\Sigma^e$ should be replaced by the state of $\Sigma$) for any constants $R>0$ and $T>0$, there exist constants $ r \in [0, R]$ and $K>0$ such that 
\begin{eqnarray}
\|x(t_k+\tau)\|\leq \|x(t_k)\|\exp({K\tau}) < R \la{expbound}
\end{eqnarray}
for all $x(t_k)\in B(0,r)$, and for all $\tau\in[0,T]$.

To prove global uniform asymptotic stability, it is necessary to show that: (1) the equilibrium point of~\re{S} is uniformly stable, and (2) that the trajectories $x(t)$, $t\geq 0$, converge to the origin uniformly with respect to time.

\underline{(1) Uniform stability:}\\
Uniform stability is proved by showing that for all $R>0$ there exists $\delta(R)>0$ such that for $\|x(t_0)\|<\delta(R)$, the state remains in $B(0,R)$, i.e. $\|x(t)\|<R$ for all $\forall\ t, t_0$, with $t\geq t_0$.  To this end, define $\delta(R) \eqbydef \beta^{-1}(\alpha(r))$.  By assumption \mbox{H\thehypcglobal}.b, $\alpha(\delta) \leq \beta(\delta) =\alpha(r)$, so $\delta \leq r$.  For all $x(t_0)\in B(0,\delta) \subset B(0,r)$, we further have that $\beta(\|x(t_0)\|)<\beta(\delta)=\alpha(r)$ because $\beta(\cdot)$ is strictly increasing.  Therefore, by assumption \mbox{H\thehypcglobal}.b: $V(x(t_0))\leq\beta(\|x(t_0)\|)<\alpha(r)$, and due to~\re{vadec}, $V(x(t_1))<V(x(t_0))<\alpha(r)$, whenever $x(t_0)\neq 0$.  Again, by assumption \mbox{H\thehypcglobal}.b, $\alpha(\|x(t_1)\|)\leq V(x(t_1))$, which implies that $\alpha(\|x(t_1)\|)< \alpha(r)$, so $\|x(t_1)\|< r$. Now, suppose that $V(x(t_n))<\alpha(r)$, for some integer $n$. If $x(t_n)\neq 0$ then, by virtue of the same argument as the one presented above, $V(x(t_{n+1}))<V(x(t_n))<\alpha(r)$, and $\alpha(\|x(t_{n+1})\|)\leq V(x(t_{n+1}))<\alpha(r)$,  so, again $\|x(t_{n+1})\| < r$. By induction, then $\|x(t_k)\|<r$ for all $k\in\natnums \cup \{0\}$, and by direct application of~\re{expbound}, $\|x(t_k+\tau)\|\leq \|x(t_k)\|\exp({K\tau})<r \exp({K\tau}) \leq R$ for all $\tau\in[0,T]$, $k\in\natnums\cup \{0\}$.  It hence follows that system $\Sigma$ with feedback law~\re{contr} is uniformly stable. 

\underline{(2) Global uniform convergence:}\\
Global uniform convergence, requires the existence of a function $\xi : \reals^n \times \reals_+ \rightarrow \reals_+$ such that, for all $x \in \reals^n$,  $lim_{t \rightarrow \infty} \xi(x,t)=0 $ and such that $\| x(t) \| \leq \xi(x(t_0), t-t_0)$ for all $t_0$, $t \geq t_0$.  The last inequality translates into the requirement that for all $R>0$ there exists $ \bar{T}(R, x_0)\geq 0$ such that $\|x(t)\|< R$, for all $t_0 > 0 $ and for all $t\geq t_0+\bar{T}$. 

Let $R >0$ be any given constant and let $\delta(R) >0$ be such that 
$\|x(t_0)\|<\delta(R)$ implies that $\|x(t)\|<R$ for all $t_0 \geq 0$ and
for all $t \geq t_0$.  Such a $\delta$ exists by virtue of uniform stability of the closed-loop system.  It remains to show that there exists an index $k^*(x(t_0),\delta) \in \natnums\cup \{0\}$ such that
\begin{eqnarray}
\|x(t_{k^*})\|<{\delta} \la{p1}
\end{eqnarray} By contradiction, suppose that $\|x(t_k)\|\geq {\delta}$, for all $k \in \natnums\cup \{0\}$. 
By virtue of~\re{vadec}, for any $k\in\natnums$: 
\vspace{-1.5em}
\begin{eqnarray}
V(x(t_{1}))-V(x(t_0))\leq-\varphi(\|x(t_{0})\|) 
                     \leq-\varphi(\delta)\nonumber\\
\vdots\hspace{35mm}\nonumber\\
V(x(t_{k}))-V(x(t_{k-1}))\leq-\varphi(\|x(t_{k-1})\|)
                         \leq-\varphi(\delta)\nonumber\\
V(x(t_{k+1}))-V(x(t_k))\leq-\varphi(\|x(t_{k})\|)
                       \leq-\varphi(\delta)\nonumber
\end{eqnarray}
Adding the above $(k+1)$ inequalities, yields
\[
V(x(t_{k+1}))-V(x(t_0))\leq-(k+1)\varphi(\delta) 
\]
which implies that 
\[
V(x(t_{k+1}))\leq V(x(t_0))-(k+1)\varphi(\delta)
              \leq \beta(\|x(t_0)\|)-(k+1)\varphi(\delta)
\hspace{1em} \forall\ k\in\natnums\cup\{0\} \nonumber
\]
The above inequality directly indicates the existence of a finite index $\bar{k} \geq 0$ such that $V(x(t_{\bar{k}}))<0$ which contradicts the fact that $V$ is positive definite. Hence, there exists a finite index $k^* \in \natnums$ such that ~\re{p1} is valid. Clearly, the index $k^*$ depends only on the value of $x(t_0)$ and $R$, but is independent of the particular value of $t_0$.

By virtue of uniform stability, we now conclude that for all $t_0 \geq 0$ and for all  $t>t_0+{k^*}T$, $\|x(t)\|< R$, which proves global uniform convergence with $\bar{T}(x_0, R)\eqbydef {k^*} T$.  This completes the proof of global uniform asymptotic stability of the closed-loop system $\Sigma$ with control~\re{contr}.
\hspace*{\fill}\qed\end{pf}

\begin{rem}
The assumption~\re{ubarbnd} is not restrictive as it can always be satisfied if the system is uniformly controllable in the following sense: for every constant extended control ${v}$ there exists a $\bar{u}\in\mathcal{P}^m$ such that $x^e(T,x,{v})=x^o(T,x,\bar{u})$ and $\|\bar{u}(\tau)\| \leq M_1 \|v\|$, for all $\tau \in [0,T]$ and some $M_1>0$.  The last requirement is not included as a design condition in SP for brevity of exposition and because it is easy to satisfy.  
\end{rem}

\section{Stabilization of an Underactuated Rigid Body in Space}
The above stabilization approach is applied to the following example in $\reals^6$ with two inputs which models an underactuated rigid-body in space:
\begin{eqnarray}
& &\hspace{0cm} \dot{x} = f_0(x)+f_1(x)u_1+f_2(x)u_2 \la{rb_rot}\\
\textrm{where,}\ \ f_0(x) &= &
\left (\sin(x_3)\,\sec(x_2)\,x_5+\cos(x_3)\,\sec(x_2)\,x_6\right ) 
      \frac{\partial}{\partial x_1}\nonumber\\
& & \hspace*{-9mm} \mbox{} +
\left (\cos(x_3)\,x_5-\sin(x_3)\,x_6\right )
      \frac{\partial}{\partial x_2}\nonumber\\
& & \hspace*{-9mm} \mbox{} +
\left (x_4+\sin(x_3)\,\tan(x_2)\,x_5+\cos(x_3)\,\tan(x_2)\,x_6\right )
      \frac{\partial}{\partial x_3} +
      a\,x_4 x_5 \frac{\partial}{\partial x_6}, \nonumber\\
  f_1(x)&=&\frac{\partial}{\partial x_4},\ \ \
  f_2(x) = \frac{\partial}{\partial x_5},\ \ \
\textrm{and}\ \ {x} = [{x}_1,\ {x}_2,\ {x}_3,\ 
                       {x}_4,\ {x}_5,\ {x}_6]^T 
\nonumber
\end{eqnarray}
with $a=-0.5$; see~\cite{COP00} for details on the model derivation.  Although more effective stabilization methods for this particular system have been constructed in the literature, these explicitly exploit the structure and the form of the system equations.  The purpose of this example is merely to explain the approach presented which applies to systems with drift, in their full generality.

The model used here is non-nilpotent, thus it does not lend itself directly to the application of the method presented.  It was however selected to indicate to the reader that the approach can be extended to non-nilpotent systems, provided that the last are suitably ``approximated'' by nilpotent systems.  At this stage, it is not our aim to introduce rigorous criteria for obtaining such approximations (see~\cite{HER91}), but rather to demonstrate that even a type of nilpotent truncation of the controllability Lie algebra of the original system can prove sufficient to implement the method.  Obviously, any such approximation or truncation should necessarily preserve controllability of the system.  It is further known, (see Theorem 2 in~\cite{LAF93}), that the steering error introduced while employing a truncated version of the controllability Lie algebra is a decreasing function of the distance between the initial and target points. It follows that the steering error can be controlled by selecting an adequately small time horizon $T$.  Both the degree of nilpotency and the horizon $T$ can be selected on a trial and error basis by requesting periodic decrease in the Lyapunov function which is a directly verifiable criterion for the adequacy of the truncation.

In the above context, system~\re{rb_rot} is assumed to be approximated by another system of a similar structure 
\[
\tilde{\Sigma}:\hspace{1cm}\dot{x}=g_0(x)+g_1(x)u_1+g_2(x)u_2
\]
whose controllability Lie algebra, $L(g_0,g_1,g_2)$, corresponds to a nilpotent truncation of order four of $L(f_0,f_1,f_2)$.  This is to say that $L(g_0,g_1,g_2)$ is nilpotent of order four and has
\[
\{g_0,\,g_1,\,g_2,\,[g_0,g_1],\,[g_0,g_2],\,[g_1,[g_0,g_2]],\,[[g_0,g_1],[g_0,g_2]]\}
\]
as its basis.  Such truncation preserves the STLC property of the system, as is easily verified using Theorem 7.3 in~\cite{SUS87}.

The differential equations~\re{geq} for the evolution of the $\gamma$-coordinates for the approximating system are:
\begin{eqnarray}
\hspace{1em}
\left [\begin{array}{c}
\dot{\gamma}_0\\\dot{\gamma}_1\\\dot{\gamma}_2\\\dot{\gamma}_3\\\dot{\gamma}_4\\\dot{\gamma}_5\\\dot{\gamma}_6
       \end{array}\right ]
& = & 
\left [\begin{array}{ccccccc} 1&0&0&0&0&0&0\\0&1&0&0&0&0&0\\0&0&1&0&0&0&0\\
0&-\gamma_0&0&1&0&0&0\\0&0&-\gamma_0&0&1&0&0\\0&\gamma_0\gamma_2&\gamma_0\gamma_1&-\gamma_2&-\gamma_1&1&0\\0&-a\gamma_0^2\gamma_2&\gamma_0\gamma_3-a\gamma_0^2\gamma_1&a\gamma_0\gamma_2&a\gamma_0\gamma_1-\gamma_3&-a\gamma_0&1
       \end{array}\right ]
\left [\begin{array}{c} u_0\\u_1\\u_2\\u_3\\u_4\\u_5\\u_6
       \end{array}\right ] \la{ge:rb_rot}
\end{eqnarray}\vspace{-2mm}\\
with $\gamma_i(0)=0$, $i=0,1,\ldots,6$.  A detailed derivation of these equations is presented in~\cite{TOR02} and makes extensive use of the package for Lie algebraic computations also described in~\cite{TOR02}.  A quadratic function $V(x)=\frac{1}{2}\|x\|^2$ is chosen to define $U^e(x)$.

Integrating~\re{ge:rb_rot} with constant controls ${v}_i$ over the interval $[0,T]$ yields the required map $F:\reals^{7}\times\reals\rightarrow\reals^6$:
{\small
\baselineskip 5mm
\begin{eqnarray}
v_1&=&F_1(\gamma^e,T)=\frac{\gamma^e_1(T)}{T}\nonumber\\
v_2&=&F_2(\gamma^e,T)=\frac{\gamma^e_2(T)}{T}\nonumber\\
v_3&=&F_3(\gamma^e,T)=\frac{1}{T}
\left (\gamma^e_3(T)+\frac{\gamma^e_0(T)\gamma^e_1(T)}{2}\right ) \la{{v}(gamma)}\nonumber\\
v_4&=&F_4(\gamma^e,T)=\frac{1}{T}
\left (\gamma^e_4(T)+\frac{\gamma^e_0(T)\gamma^e_2(T)}{2}\right )\nonumber\\
v_5&=&F_5(\gamma^e,T)=\frac{1}{T}
\left (\gamma^e_5(T)+\frac{\gamma^e_1(T)\gamma^e_4(T)}{2}+\frac{\gamma^e_2(T)\gamma^e_3(T)}{2}-\frac{\gamma^e_0(T)\gamma^e_1(T)\gamma^e_2(T)}{6}
\right )\nonumber\\
v_6&=&F_6(\gamma^e,T)=\frac{1}{T}
\left
(\gamma^e_6(T)+\frac{a\,\gamma^e_0(T)\gamma^e_5(T)}{2}+\frac{\gamma^e_3(T)\gamma^e_4(T)}{2}
+\frac{a\,{\gamma^e_0}^2(T)\gamma^e_1(T)\gamma^e_2(T)}{12}\right .\nonumber\\
&&\hspace{7mm}\left .-\frac{\gamma^e_0(T)\gamma^e_2(T)\gamma^e_3(T)}{12}(1+a)
+\frac{\gamma^e_0(T)\gamma^e_1(T)\gamma^e_4(T)}{12}(1-a)
\right )
\nonumber
\end{eqnarray}
}
The above results in the following expression for the reachable set
in the $\gamma$-coordinates: 
\[
\mathcal{R}_{\gamma}(T,U^e(x))=
\left \{ \gamma~|~\sum_{i=0}^{6} \nabla{V}g_i(x)F_i(\gamma,T) < -\eta\|x\|^2,\ \|F(\gamma,T)\|\leq M\|x\| \right \} 
\]
The expressions for $\gamma(T,\bar{u}^d)$ in terms of the constant parameters defining $\bar{u}^d \in\mathcal{P}^{r}$ are easily obtained by symbolic integration and are subsequently employed to solve SP.  The values  $s=6$, $T=0.1$, $\eta=1$, $M=10$, $R=2$ and $C=50$, were assumed in the solution of SP.   The simulation results are obtained using an initial condition $x_0=[-0.1\ 0\ 0.2\ 0\ 0\ 0.1]^T$ and are shown in Figure~\ref{fig:rb2a}. 
\begin{figure}[htbp]
\hspace*{-2mm}
\begin{tabular}{cc}
  \includegraphics[scale=0.37]{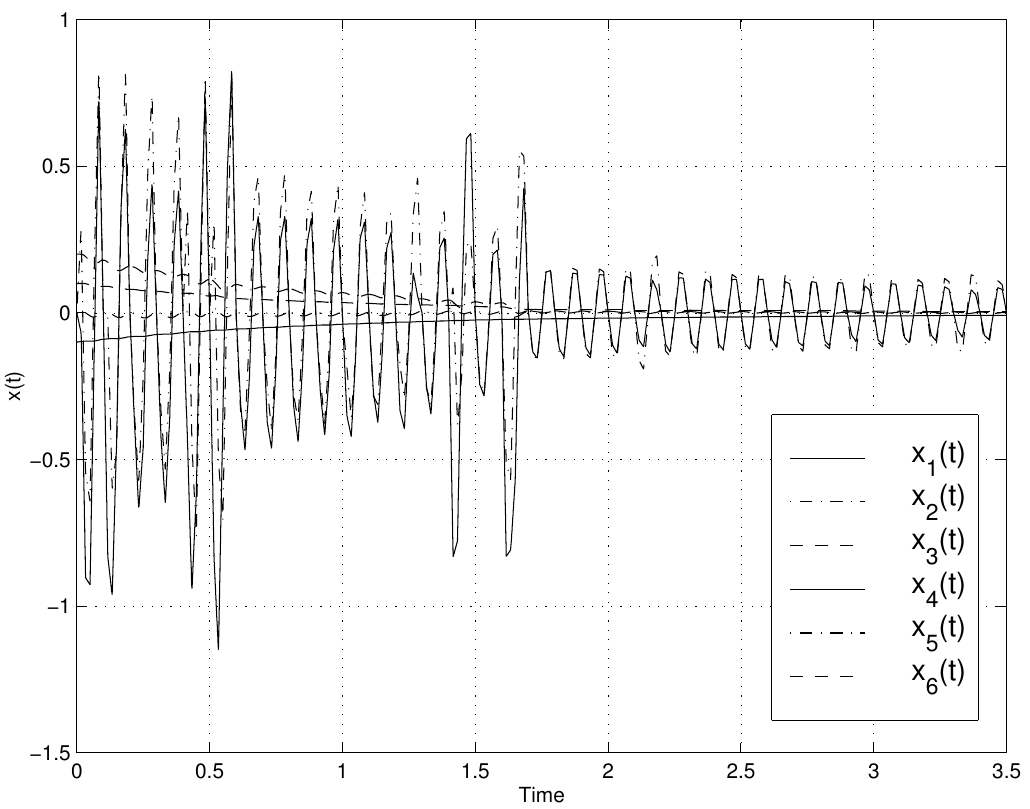}\vspace{-2ex}&
  \includegraphics[scale=0.37]{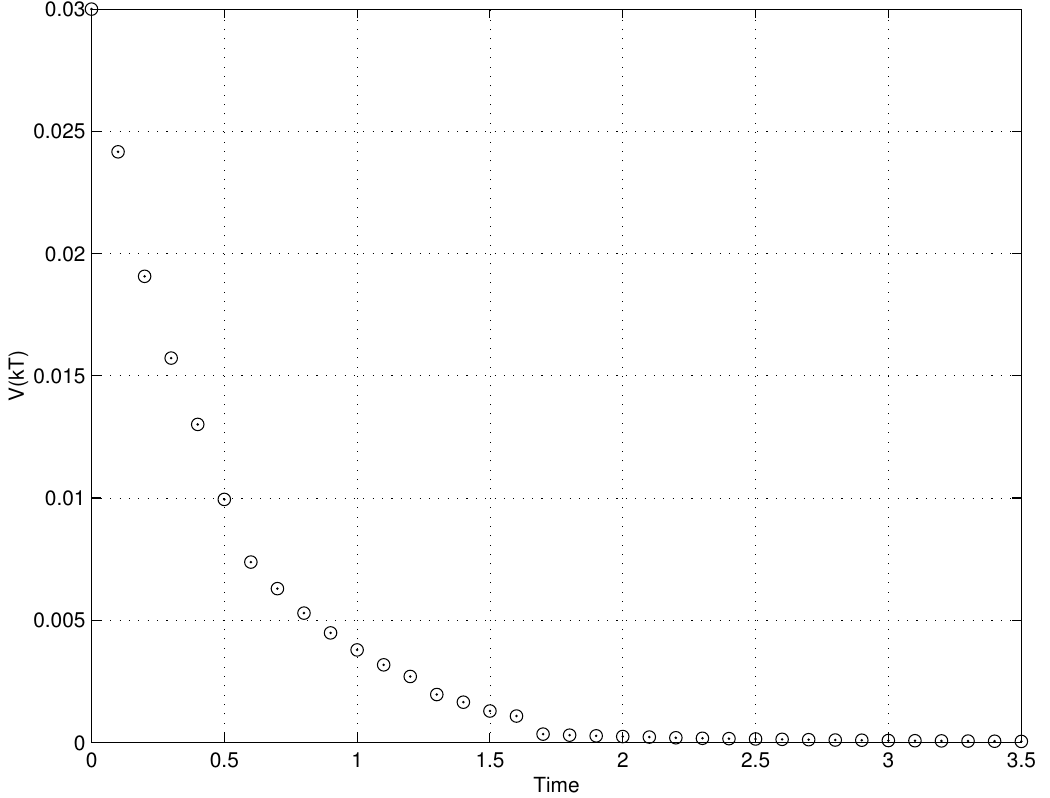}\vspace{-2ex}\\[1em]
 a & b\\
 {\small (a) State trajectory $x(t)$.} & 
 \hspace{13mm}\parbox[t]{0.4\textwidth}
 {\small (b) Function $V(x(t))$\\
     \hspace*{1.5em} for $t=kT$, $k=0,1,\ldots, 35$.}
\end{tabular}
\caption{Results for the stabilization of the rigid body
  obtained with the concatenated control $u^c(x,\tau)$.}
\label{fig:rb2a}
\end{figure}

\section{Conclusions}
The presented stabilizing feedback design approach applies to general systems with drift for which controllable linearizations, as well as continuous stabilizing state feedback laws, may not exist.

The approach is computationally expensive, but less so as compared with~\cite{MIC01}.  As compared with the control procedure of~\cite{HER80}, it provides a more systematic tool for generation of complicated Lie bracket motions of the system which might be necessary in the process of stabilization.

The approach does not deliver the feedback uniquely as it employs arbitrary solutions to a satisficing nonlinear programming problem.  This can be viewed as the strength of the method, as it leaves the designer much freedom to accommodate for other goals.  For example, it may be desired to solve SP while simultaneously minimizing the number of discontinuities in the open loop control or else to construct a time-varying continuous control.  

Future research is aimed at exploring possible simplifications which might originate from a particular structure of the Wei-Norman equations describing the evolution of the flow of the system on the Lie group.
\appendix
\renewcommand{\theequation}{\thesection.\arabic{equation}}

\section{Proof of Proposition~\ref{prop:2}}\label{appen1}
Since $V$ is twice continuously differentiable, $g^v$ is analytic and linear in ${v}$, and $g_0(0)=0$, then $\nabla{V}$ and $g^v$ are Lipschitz continuous on $B(0,2R)$, uniformly with respect to  ${v}=[{v}_1,\ldots,{v}_{r-1}]^T$ satisfying  $\|{v}\| \leq M \|x\|$.  Hence, there exists a $K>0$ such that:
\begin{eqnarray}
          \left \|g^w(y)-g^v(x)\right \| \leq 
             K\|y-x\| 
\hspace{1em}\mbox{ and }\hspace{1em}
          \left \|\nabla{V}(y)-\nabla{V}(x)\right \| \leq 
             K\|y-x\| \nonumber
\end{eqnarray}
for all $x\in B(0,R)$, $y\in B(0,2R)$, and for any constant $v$ and  $w$ such that $\|v\| \leq M \|x\|$ and $\|w\| \leq M \|x\|$.

Let $x^e(t) \eqbydef x^e(t,x,v)$, $t \geq 0$. First, it is shown that there exists a $T_1>0$ and a constant $K_1>0$ such that
\begin{eqnarray}
\|x^e(s)-x\| &\leq& \|x\|
         \left ( \exp({K\,s})-1\right ) \la{bdx}
\end{eqnarray}
and
\begin{eqnarray}
\|x^e(s)\|\leq K_1\|x\| \la{bxs}
\end{eqnarray}
for all $s\in [0,T_1]$ such that $x^e(s)\in B(0,2R)$.

To this end it suffices to notice that
\begin{eqnarray}
\|x^e(s)-x\|
           & \leq& \int_{0}^{s}\| g^v(x)\| {\d}\tau +
                   \int_{0}^{s}\| g^v(x^e(\tau))-g^v(x)\| {\d}\tau\nonumber\\ 
           & \leq& K\,\|x\|\,s
                   +\int_{0}^{s}K\|x^e(\tau)-x\| {\d}\tau\nonumber
\end{eqnarray}
which, by the application of the Gronwall-Bellman lemma, yields inequality~\re{bdx}.

It can be shown that if $T_{1}$ is chosen so that $\left ( \exp({K\,T_{1}})-1\right )\leq  \frac{1}{2}$ then~\re{bdx} holds for $s\in [0,T_1]$.  By contradiction, suppose that there exists an $s_1 < T_1$ such that $\|x^e(s_1)\| = 2R$.  It follows that $2R \leq \|x\| + \|x^e(s_1) - x\| \leq R + \|x\| \left ( \exp({K\,s_1})-1\right ) \leq \frac{3}{2}R$ which is false, and hence~\re{bdx} is valid for $s\in [0,T_1]$.  Inequality~\re{bxs} follows from ~\re{bdx} since
\begin{eqnarray}
\|x^e(s)\|\,\leq\, \|x^e(s)-x\|+\|x\|\,\leq\, \|x\|\exp({K\,s}) \,\leq\, K_1\|x\|\nonumber
\end{eqnarray}
with $K_1=\exp({KT_1})$.

Now,
{\small
\begin{eqnarray}
\hspace*{-5mm}  V(x^e(T))-V(x) 
        &\leq& \nabla{V}(x)g^v(x)T + 
               \int_{0}^{T}\left \|\nabla{V}(x^e(s))g^v(x^e(s))
                                  -\nabla{V}(x)g^v(x)\right \| {\d}s
                                                              \nonumber\\[-1.3ex]
&& \la{DV2}\\[-1.3ex]
        &\leq& -\eta\, \|x\|^2\, T + 
          \bar{K}\|x\| \int_{0}^{T}\left \|x^e(s) - x\right \| {\d}s\nonumber
\end{eqnarray}
}
because 
{\small
\begin{eqnarray}
\left\|\nabla{V}(x^e(s))g^v(x^e(s))-\nabla{V}(x)g^v(x)\right \| &\leq&
     \|\nabla{V}(x^e(s))g^v(x^e(s))-\nabla{V}(x^e(s))g^v(x)\|\nonumber\\
& &\mbox{} 
   + \|\nabla{V}(x^e(s))g^v(x)-\nabla{V}(x)g^v(x)\|\nonumber\\
&\leq& \bar{K}\|x\|\,\|x^e(s)-x\| \nonumber
\end{eqnarray}
}
with $\bar{K}=K^2(K_1+1)$.

Hence, if $T<T_1$ then $x^e(s)\in B(0,2R)$ for all $s\in [0,T]$ and, using~\re{bdx} in~\re{DV2}, yields 
{\small 
\[
\hspace*{-5mm}V(x^e(T))-V(x)\ \leq\ -\eta\, \|x\|^2\,T + 
     \bar{K} \|x\|^2 \int_{0}^{T} \left ( \exp({K\,s})-1 \right ) {\d}s
     \ \leq\  -\frac{\eta}{2}\|x\|^2 q(T)\nonumber
\]
}
where $q(T) \eqbydef \left (2 + \frac{2\,\bar{K}}{\eta}\right )T-\frac{2\,\bar{K}}{\eta\,K} \left (\exp({K\,T})-1\right )$.  If $r(T)\eqbydef q(T)-T$, then  $r(0)=0$ and $r'(0)=1$, so there exists a $T_{max}\leq T_1$ such that $r(T)\geq 0$ for all $T\in [0,T_{max}]$.  Hence $q(T)\geq T$ for all $T\in [0,T_{max}]$ which proves~\re{decVe}.\hfill$\square$

\section{Proof of Proposition~\ref{prop:3}}\label{appen2}
For any $\epsilon\in (0,\frac{\eta}{\zeta^2}]$ and any given $x\in B(0,R)$ 
let 
\[
z(x)\eqbydef-\epsilon \nabla{V}^T(x)
\]
Then, by virtue of hypothesis \mbox{H\thehypcglobal}.b, $z(x)$ satisfies 
\[
\nabla{V}z(x)=-\epsilon\|\nabla{V}(x)\|^2\leq -\epsilon\zeta^2\|x\|^2\leq 
  -\eta\|x\|^2
\]
and is realizable as the right-hand side of the extended system~\re{Se}, i.e. there exists an extended control ${v}$ such that
\[
 g_0(x)+\sum_{i=1}^{r-1}g_i(x){v}_i = z(x)
\]
Clearly,
\[
 {v}(x) = Q^{\dagger}(x)\left (z(x)-g_0(x)\right ),\hspace{1cm} 
 {v}=[{v}_1\ {v}_2\ \ldots\ {v}_{r-1}]^T 
\]
where\hfill $Q^{\dagger}=Q^T\left (Q\,Q^T\right )^{-1}$\hfill is\hfill the\hfill pseudo-inverse\hfill of\hfill the\hfill $n\times (r-1)$\hfill matrix\hfill $Q(x)=$\newline $[g_1(x)\ g_2(x)\ \ldots\ g_{r-1}(x)]$, which is guaranteed to exist for all $x\in\reals^n$ because $\mathrm{rank}\left (\rule[-0.05em]{0em}{1.1em}Q(x)\right )=n$ by construction of the extended system~\re{Se}.  Moreover, $Q^{\dagger}$ is a smooth matrix function of $x$, thus there exists a constant $c(R)>0$ such that 
\[
\|Q^{\dagger}(x)\|\leq c,\hspace{1cm}\forall\ x\in B(0,R)
\]
By Lipschitz continuity of $g_0$ and $\nabla{V}$, there exist constants $d(R)>0$ and $K(R)>0$ such that:\vspace{-6mm}
\begin{eqnarray}
\|{v}\|&\leq& \|Q^{\dagger}(x)\|\,\|z(x)-g_0(x)\|\nonumber\\
       &\leq& \|Q^{\dagger}(x)\|\,\left (\|z(x)\|+\|g_0(x)\|\right )\nonumber\\
       &\leq& c\left (\frac{\eta\,K}{\zeta^2}\|x\|+d\|x\|\right )\nonumber  
\end{eqnarray}
and hence with $M=c\left (\frac{\eta\,K}{\zeta^2}+d\right )$, the extended control ${v}$ is in the set $U^e(x)$; it is only one of many controls which satisfy ${v}\in U^e(x)$.  Let $\gamma^e(T)$ be the $\gamma$-coordinates of the flow at time $T$ of the extended system~\re{Se} with control ${v}$.  Then $\gamma^e(T)\in \mathcal{R}_{\gamma}(T,U^e(x))$.  By virtue of the controllability assumption~\ref{h:controllable}, there exists an open loop control $\bar{u}\in\mathcal{P}^m$ such that $\gamma^e(T,{v}^d)=\gamma(T,\bar{u}^d)$.  Hence, $\gamma(T,\bar{u}^d)\in \mathcal{R}_{\gamma}(T,U^e(x))$ thus proving the existence of solutions to SP at any $x\in B(0,R)$.\hfill$\square$
\input{abbrev}

\end{document}

%% file: abbrev.tex
\def\AMO{Appl. Math. Optim.}
\def\AUTOM{Automatica J. IFAC}

\def\COMMPAM{Commun. Pure Appl. Math.} 
\def\COMPPHYS{Comput. Phys. Comm.}

\def\EJC{Eur. J. Control}

\def\IEEEAC{IEEE Trans. Automat. Control}
\def\IEEEIP{IEEE Trans. Image Process.}
\def\IEEEIT{IEEE Trans. Inform. Theory}
\def\IEEESP{IEEE Trans. Signal Process.}

\def\IMHOTEP{IMHOTEP J. Afr. Math. Pures Appl.}

\def\IJC{Internat. J. Control}

\def\JDE{J. Differential Equations}
\def\JFA{J. Funct. Anal.}
\def\MCSS{Math. Control Signals Systems}

\def\AMSM{Mem. Amer. Math. Soc.}

\def\SIAMCON{SIAM J. Control Optim.}
\def\SIAMREV{SIAM Rev.}
\def\SYSCON{Systems Control Lett.} 

\def\ZWV{Z. Wahrsch. Verw. Gebiete} 

\def\ACC{American Control Conf.}
\def\CDC{IEEE Conf. on Decision and Control}
\def\ICIP{IEEE Int. Conf. on Image Process.}
\def\IFAC{IFAC Proc. Ser.}
\def\AMSP{Proc. Amer. Math. Soc.}
\def\IEEEP{Proc. of the IEEE}